\documentclass[a4paper,english,oneside,11pt]{smfcustom}
\usepackage{mathrsfs,mathtools}
\usepackage[english]{babel}
\usepackage[utf8]{inputenc}
\usepackage[T1]{fontenc}
\usepackage{pxfonts,microtype,enumitem,framed,xspace}
\usepackage{xcolor}\usepackage[all]{xy}
\usepackage{etex}
\usepackage[labelfont=normal]{subfig}
\usepackage{tikz}
\usetikzlibrary{patterns,arrows}
\def\C{\mathbb{C}} \def\P{\mathbb{P}} \def\Q{\mathbb{Q}}   \def\N{\mathbb{N}}
 \def\O{\mathcal{O}}

\def\z{\boldsymbol{z}}

\def\bydef{\coloneqq}

\makeatletter
\DeclareRobustCommand\diff{\@ifnextchar[{\@@diff}{\@diff}}
\def\@diff{\mathop{}\!{\mathrm{d}}}
\def\@@diff[#1]{\cramped{\diff^{#1}}} 
\DeclareRobustCommand\Diff{\@ifnextchar[{\@@Diff}{\@Diff}}
\def\@Diff{\mathop{}\!{\mathrm{D}}\cdot}
\def\@@Diff[#1]{\cramped{\mathop{}\!{\mathrm{D}^{#1}}}\cdot} 
\makeatother
\DeclareMathOperator{\mult}{mult}
\DeclareMathOperator{\Bs}{Bs}
\DeclareMathOperator{\BB}{\mathbf{B}}
\DeclareMathOperator{\rk}{rank}

\DeclareMathOperator{\Aut}{Aut}

\DeclarePairedDelimiter{\abs}{\lvert}{\rvert} 
\DeclarePairedDelimiter{\Set}{\{}{\}}

\def\myfrac#1#2{
  \mathchoice
  {\raisebox{.3ex}{\scalebox{.9}{\(#1\)}}/\raisebox{-.3ex}{\scalebox{.9}{\(#2\)}}}
  {\raisebox{.3ex}{\scalebox{.9}{\(#1\)}}/\raisebox{-.3ex}{\scalebox{.9}{\(#2\)}}}
  {\raisebox{.1ex}{\(\scriptstyle #1\)}/\raisebox{-.1ex}{\(\scriptstyle #2\)}}
  {\raisebox{.1ex}{\(\scriptscriptstyle #1\)}/\raisebox{-.1ex}{\(\scriptscriptstyle #2\)}}
}
\def\T_#1{{\Omega^{\mathrlap{\vee}}}_{#1}}
\let\leq\leqslant\let\geq\geqslant

\NoSwapTheoremNumbers
\newtheorem{alphtheo}{\theoname}%

\newenvironment{ques}{\begin{enonce*}[remark]{Question}}{\end{enonce*}}
\title[Quasi-positive orbifold cotangent bundles]{Quasi-positive orbifold cotangent bundles ;\\pushing further an example by Junjiro Noguchi.}
\author{Lionel Darondeau}
\address{IMAG, Univ Montpellier, CNRS, Montpellier, France.}
\email{lionel.darondeau@normalesup.org}
\author{Erwan Rousseau}
\address{Institut Universitaire de France \& Aix Marseille Univ, CNRS, Centrale Marseille, I2M, Marseille, France.}
\email{erwan.rousseau@univ-amu.fr}

\thanks{E. R. was partially supported by the ANR project \lq\lq FOLIAGE\rq\rq{}, ANR-16-CE40-0008.}
\date{\today{} (version 2)}
\begin{document}
\begin{abstract}
  In this work, we investigate the positivity of logarithmic and orbifold cotangent bundles along hyperplane arrangements in projective spaces. We show that a very interesting example given by Noguchi (as early as in 1986) can be pushed further to a very great extent. Key ingredients of our approach are the use of Fermat covers and the production of explicit global symmetric differentials. This allows us to obtain some new results in the vein of several classical results of the literature on hyperplane arrangements. These seem very natural using the modern point of view of augmented base loci, and working in Campana's orbifold category.
  As an application of our results, we derive two new orbifold hyperbolicity results, going beyond some classical results of value distribution theory.
\end{abstract}
\maketitle

\section{Introduction}
\subsection*{Positive and quasi-positive cotangent bundles}
In recent years, families of varieties with ample cotangent bundles have attracted a lot of attention (see \textit{e.g.} \cite{Deb05,Xie18,BDa18,Den20,mohsen,CR20,Etesse}), and there have been significant progress in this area (even though finding an explicit surface with ample cotangent bundle in \(\P^4\) is still a tremendous challenge).
With the development of our understanding, the enriching of techniques, and in connection with hyperbolicity problems, some variations of this problem have started to emerge.
For instance, in \cite{BDe18}, the authors have been interested in the determination of the augmented base locus of logarithmic cotangent bundles along normal crossing divisors in projective spaces.
The \textsl{stable base locus} \(\BB(L)\subseteq X\) of a line bundle \(L\) on a projective variety \(X\) is defined as the intersection of the base loci of all multiples of \(L\).
Then, the \textsl{augmented base locus} (or \textsl{non-ample locus}) \(\BB_{+}(L)\subseteq X\) is
\[
  \BB_{+}(L)
  \bydef
  \bigcap_{q\in\N}
  \BB(qL-A),
\]
for any ample line bundle \(A\to X\).
The augmented base locus of a line bundle is a geometric measure of the positivity of its sheaf of global sections. In particular, it is different from the base variety when the line bundle is big, and it is empty when the line bundle is ample.
For vector bundles, one studies the augmented base locus of the Serre line bundle on their projectivizations.
The idea of augmented base loci for vector bundles can be traced back to~\cite{Nog77}, where it was already used in connection to hyperbolicity (see below).

In various cases, one does not really need this augmented base locus to be empty in order to obtain interesting geometric consequences, and in many interesting settings, such as logarithmic and orbifold setting, one actually cannot expect the augmented base locus to be empty. This leads to the definition of several notions of positivity, where one only ask for a certain geometric control of the non-ample locus.  For example a line bundle \(L\) is said to be \textsl{ample modulo a divisor \(D\)} when \(\BB_{+}(L)\subseteq D\).  Note that if \(L\) is ample modulo \(D\), then \(L\) is necessarily big.

Denote by \(p\colon X'\bydef\P(\Omega(X,D))\to X\) the projectivized bundle of rank \(1\) quotients of the logarithmic cotangent bundle of a smooth logarithmic pair \((X,D)\). In this work, we will use the following definition.
\begin{defi}
  \label{defi:quasiample}
  We say that the cotangent bundle of \((X,D)\) is \textsl{ample modulo boundary} if
  \[
    p\big(\BB_+(\O_{X'}(1))\big)\subseteq D.
  \]
\end{defi}

It is a weaker positivity property that the one introduced in~\cite{BDe18}. 
Consider the various residue exact sequences coming with a simple normal crossing divisor in \(\P^n\). One gets a lot of trivial quotients supported on the boundary components. Then, the projectivizations of these trivial quotients give subvarieties in the projectivized logarithmic cotangent bundle, that constitute obstructions to the ampleness of the logarithmic cotangent bundle (see~\cite[Sect.~2.3]{BDe18}).
In particular, one has always
\[
  D\subseteq p\big(\BB_+(\O_{X'}(1))\big).
\]
One can  hence view Definition~\ref{defi:quasiample} as asking the projection of the augmented base locus to be minimal.
Brotbek and Deng define \(\Omega(X,D)\) to be ``almost ample'' when the augmented base locus \(\BB_+(\O_{X'}(1))\) itself (and not its projection) is minimal. This means that the augmented base locus corresponds exactly to the trivial quotients of the cotangent bundle given by the residue short exact sequence.
Then, one has the following (\cite[Theo.~A]{BDe18}).
\begin{theo}[Brotbek--Deng]
  \label{theo:BDe}
  Let \(Y\) be a smooth projective variety of dimension \(n\), with a very ample line bundle \(H\to Y\).
  For \(c\geq n\) and \(d\geq(4n)^{n+2}\), the logarithmic cotangent bundle along the sum \(D=D_{1}+\dotsb+D_{c}\) of \(c\) general hypersurfaces \(D_{1},\dotsc,D_{c}\in\abs{H^{d}}\) is ``almost ample''.
\end{theo}
This result is optimal concerning the number of components of the boundary divisor (\cite[Prop.~4.1]{BDe18}).
\begin{prop}[Brotbek--Deng]
  The logarithmic cotangent bundle along a simple normal crossing divisor with \(c<n\) irreducible components in \(\P^{n}\) is never big. 
\end{prop}

The effective degree bounds in~\cite{BDe18} being quite large, it is a natural question to ask what would be the optimal degree bound (when one relaxes the condition on the number of components) ?
An associated problem is to find some low degree examples of pairs with ample cotangent bundles modulo boundary.

To the best of our knowledge, before~\cite{BDe18}, the only example of such quasi-positivity of the cotangent bundle is due to Noguchi~\cite{Nog86}. It is an example given as early as in 1986, in the paper in which he defined logarithmic jet bundles. Noguchi introduced the following positivity property.

\begin{defi}
  Let \((X,D)\) be a smooth logarithmic pair. Denote \(V\bydef X\setminus D\). A vector bundle \(E\) on \(X\) is said ``quasi-negative'' over \(V\) if there is a proper morphism \(\varphi\colon E\to\C^{N}\) to an affine space, such that \(\varphi\) is an isomorphism from \(E\rvert_{V}\setminus O\) to \(\varphi(E)\setminus\varphi(E\rvert_{D})\), where \(O\) denotes the zero section.
\end{defi}
Then, one has the following (\cite{Nog86}).
\begin{theo}[Noguchi]
  The logarithmic tangent bundle along a general arrangement \(\mathscr{A}\) of \(6\) lines in \(\P^{2}\) is ``quasi-negative'' over \(\P^{2}\setminus\mathscr{A}\).
\end{theo}
The rough idea of the proof is that using an explicit basis of the logarithmic cotangent sheaf along an arrangement of lines in general position, one is able to construct an immersive Kodaira map (under some further explicit genericity condition).
Some combinatorial work allows one to identify this supplementary genericity condition as asking that in the dual projective space parametrizing hyperplanes, the points of the arrangement do not all lie in a single quadric.

Now, one has:
\begin{lemm}
  \label{lemm:bplus}
  Let \(L\) be a globally generated line bundle.
  If \(\abs{L}\) defines an immersive map on \(X\setminus V\), then \(\BB_+(L)\subseteq V\).
\end{lemm}
\begin{proof}
  According to \cite[Theo.~A]{BCL14}, the augmented base locus \(\BB_{+}(L)\) is the smallest closed subset \(V\) of \(X\) such that the linear system \(\abs{qL}\) defines an isomorphism of \(X\setminus V\) onto its image for sufficiently large \(q\).

  For \(q\) large enough, the Stein factorization of \(\phi_{\abs{L}}\) is given by
  \[
    X
    \stackrel{\abs{qL}}\longrightarrow
    \phi_{\abs{qL}}(X)
    \stackrel{\nu_q}\longrightarrow
    \phi_{\abs{L}}(X),
  \]
  for some finite morphism \(\nu_q\) (\cite[Lemma~2.1.28]{Laz}).
  Now, since \(\abs{L}\) defines an immersive map on \(X\setminus V\), the fibers of \(\phi_{\abs{L}}\) are discrete.
  An immediate consequence is that on this set \(\phi_{\abs{qL}}\) has discrete and connected fibers.
  In other words, for sufficiently large \(q\) the linear system \(\abs{qL}\) defines an isomorphism of \(X\setminus V\) onto its image.
\end{proof}
This lemma allows us to reformulate the result of Noguchi as follows.
\begin{theo}[Noguchi]
  \label{theo:noguchi}
  The logarithmic cotangent bundle along an arrangement \(\mathscr{A}\) of \(d\geq6\) lines in \(\P^{2}\) in general position with respect to hyperplanes and to quadrics is ample modulo \(\mathscr{A}\).
\end{theo}

As mentioned above, in this smooth logarithmic setting, one cannot expect the orbifold cotangent bundle to be plainly ample, and we have explained that ampleness modulo boundary is somehow optimal. 
Concerning the optimal number of lines, combining Noguchi's result with Theorem~\ref{theo:optimal} below, we now that it can only be \(5\) or \(6\). It is not clear yet how to prove that for \(5\) lines one cannot expect ampleness modulo boundary of the logarithmic cotangent bundle.

\subsection*{Hyperbolicity of complements of hypersurfaces}
A very connected research area is the one of complex hyperbolicity. Indeed, the following result is now classical (see~\cite{Nog77} for the compact case). Given a logarithmic symmetric differential form \(\omega\) on a smooth logarithmic pair \((X,D)\), which vanishes on an ample divisor, all entire maps \(f\colon\C\to X\setminus D\) lands in the zero locus of \(\omega\). 
In other words, \(f(\C)\subseteq p(\BB_+(\O_{X'}(1)))\).
If \(\Omega(X,D)\) is ample modulo boundary, one immediately gets that all these curves are constant. One says that the pair \((X,D)\) is \textsl{Brody hyperbolic}.
We see that here, there is no need to have global ampleness in order to obtain interesting geometric applications.

It is thus an interesting companion question to ask about the hyperbolicity of complements of hypersurfaces. Concerning this question, a very interesting setting seems to be the classical setting of hyperplane arrangements, for which optimal degree bounds are reached. 

To sum up some classical results of value distribution theory:
in the case of hyperplane arrangements, the (conjectural) optimal degree bounds are reached.

\begin{conj}[Kobayashi]
  The complement of a general high degree hypersurface in \(\P^{n}\) is Brody-hyperbolic.
\end{conj}
\begin{theo}[Zaidenberg~\cite{Zai87,Zai93}]
  For a general hypersurface \(D\) in \(\P^{n}\) of degree \(2n\), there is a line in \(\P^{n}\) meeting \(D\) in at most two points.
\end{theo}
\begin{theo}[Bloch, Cartan, Green~\cite{Gre72}]
  The complement of an arrangement of \(2n+1\) hyperplanes in general position in \(\P^{n}\) is Brody-hyperbolic.
\end{theo}
\begin{theo}[Snurnitsyn~\cite{Snu86,Zai93}]
  \label{theo:optimal}
  For any arrangement \(\mathscr{A}\) of \(2n\) hyperplanes in \(\P^{n}\), there is a line in \(\P^{n}\) meeting \(\mathscr{A}\) in only two points.
\end{theo}

And these results have also their counterparts concerning weak hyperbolicity.
\begin{conj}[Green--Griffiths--Lang]
  On a logarithmic pair \((X,D)\) of logarithmic general type, there is a proper subvariety \(\mathrm{Exc}(X)\subsetneq X\) containing the images of all non-constant entire maps \(f\colon\C\to X\setminus D\).
\end{conj}
\begin{theo}[Borel,Green~\cite{Gre72}]
  The maps \(f\colon\C\to\P^{n}\setminus\mathscr{A}\) with values in the complement of an arrangement of \(n+2\) hyperplanes are linearly degenerate.
\end{theo}
Here, the condition \(d\geq n+2\) corresponds exactly to the general type assumption.
Notice that this is a refinement of the classical theorem of Borel, since there is no genericity assumption in the statement.

This motivates us to work with the interesting setting of hyperplane arrangements in the rest of the paper.
An underlying question is the following:
Does the optimal lower bound on the degree of hyperplane arrangement for which the logarithmic cotangent bundle is ample modulo boundary provide indications on a lower bound on the degree for which Theorem~\ref{theo:BDe} should hold?

Most of the results on hyperbolicity of complements of hyperplanes are obtained using Nevanlinna's theory of value distribution (see e.g.~\cite{Kob98} or~\cite{NW}). One of the key tools is so-called Cartan's Second Main Theorem, which allows one to study not only entire curves in complements but also entire curves intersecting the boundary divisor with prescribed multiplicities (see e.g.~\cite[Coro~3.B.46]{Kob98}). A complementary modern point of view on these \textsl{orbifold curves} is also given by the theory of Campana's orbifolds~\cite{Campana}. An alternative approach to Nevanlinna theory for orbifold hyperbolicity  is developed in \cite{CDR20}. We will pursue these ideas here, studying the augmented base loci of orbifold cotangent bundles along hyperplane arrangements. In this direction, to the best of our knowledge, there are no existing results in the literature before this work.

\subsection*{Main results of the paper}
The common thread of this work is to push further Theorem~\ref{theo:noguchi}.
We obtain three main new results in this direction (Theorems~\ref{theo:lognog}, \ref{theo:orbinog}, \ref{theo:orbibig}).
Then, we derive two new hyperbolicity results (Theorems~\ref{theo:fermathyp} and~\ref{theo:orbihyp}).

We generalize the result of Noguchi to higher dimensions. We prove the following.
\begin{alphtheo}
  The logarithmic cotangent bundle along an arrangement \(\mathscr{A}\) of 
  \(d\geq\binom{n+2}{2}\)
  hyperplanes in \(\P^{n}\) in general position with respect to hyperplanes and to quadrics is ample modulo \(\mathscr{A}\).
\end{alphtheo}

We extend the result of Noguchi to the geometric orbifold category introduced by Campana.
\begin{alphtheo}
  The orbifold cotangent bundle along an arrangement \(\mathscr{A}\) of 
  \(d\geq \binom{n+2}{2}\) 
  hyperplanes in \(\P^{n}\) in general position with respect to hyperplanes and to quadrics,
  with multiplicities \(m\geq 2n+2\),
  is ample modulo \(\mathscr{A}\).
\end{alphtheo}
Theorem~\ref{theo:noguchi} amongs to \(n=2\) in Theorem~\ref{theo:lognog}.
Theorem~\ref{theo:lognog} constitutes the case of infinite orbifold multiplicity in Theorem~\ref{theo:orbinog}.

Lastly,  we prove the positivity of orbifold cotangent bundles in all dimensions with low degrees and very low multiplicities.
\begin{alphtheo}
  For \(n\geq 2\), the orbifold cotangent bundle along an arrangement \(\mathscr{A}\) of \(d\geq 2n(\frac{2n}{m-2}+1)\) hyperplanes in \(\P^{n}\) with multiplicity \(m\geq 3\) is big.
\end{alphtheo}
Theorem~\ref{theo:orbibig} is weaker concerning positivity of the cotangent bundle but is spectacular concerning multiplicities. Remark also that taking \(m\) linear in \(n\),  one gets a linear lower bound on the degree.

Next, we derive two hyperbolicity results from Theorem~\ref{theo:orbinog} (and from its reformulation in terms of Fermat covers).
The first result is in the vein of several classical results in the literature on Fermat covers (see~\cite{Kob98,Dem97}). 
\begin{alphtheo}
  The Fermat cover associated to an arrangement \(\mathscr A\) of \(d\geq\binom{n+2}{2}\) hyperplanes in \(\P^ n\) in general position with respect to hyperplanes and to quadrics, with ramification \(m\geq 2n+2\), is Kobayashi-hyperbolic.
\end{alphtheo}

The second result could be seen as a strong hyperbolicity counterpart of the classical (weak) hyperbolicity results derived from Cartan's Second Main Theorem.
\begin{alphtheo}
  Consider an arrangement \(\mathscr A\) of \(d\) hyperplanes \(H_1,\dotsc,H_d\) in \(\P^n\) in general position with respect to hyperplanes and to quadrics, with respective orbifold multiplicities \(m_i\), and the associated orbifold divisor \(\Delta\bydef\sum_{i=0}^d(1-1/m_i)\cdot H_i\).
  If \(d\geq\binom{n+2}{2}\) and \(m_i\geq 2n+2\), the orbifold pair \((\P^n,\Delta)\) is Kobayashi-hyperbolic.
\end{alphtheo}

It is noteworthy that, even if both results are in the same flavor, Theorem~\ref{theo:orbihyp} is \emph{not} a consequence of Theorem~\ref{theo:fermathyp}.
To the best of our knowledge, both hyperbolicity results are new.

\subsection*{Organization of the paper}
The paper is organized as follows.
In §\ref{se:lognog}, we generalize the result of Noguchi to higher dimensions and prove Theorem~\ref{theo:lognog}, using an explicit cohomological method, in the spirit of the original approach by Noguchi.

In §\ref{se:campana}, we introduce precise definitions for various notions of positivity of orbifold cotangent bundles.

In §\ref{se:orbinog}, we extend the result of Noguchi to the orbifold category introduced by Campana and prove Theorem~\ref{theo:orbinog}, using a quite different explicit cohomological method. We rephrase the approach of explicit \v{C}ech cohomology on complete intersections by Brotbek in the context of what we call Fermat covers. Computations would tend to be quickly intractable when dimension grow. However, we are able to use the assumption of general type with respect to quadrics brought out in the study of the logarithmic case in order to tame a little the computations and find a quick way to the proof. 

In §\ref{se:orbibig}, we investigate the existence of orbifold symmetric forms for low multiplicities and prove Theorem~\ref{theo:orbibig}, using a non-explicit cohomological method. We derive the sought result from works by Brotbek and by Coskun--Riedl, using again Fermat covers.

In §\ref{se:hyperb}, we focus on hyperbolicity questions and we prove Theorems~\ref{theo:fermathyp} and~\ref{theo:orbihyp}, building on the results of Sect.~\ref{se:orbinog}.

\setcounter{alphtheo}{0} 
\section{Ampleness modulo boundary of the logarithmic cotangent bundle}
\label{se:lognog}
This section is devoted to prove the following generalization of Noguchi's example.
\begin{alphtheo}
  \label{theo:lognog}
  The logarithmic cotangent bundle along an arrangement \(\mathscr{A}\) of \(d\geq\binom{n+2}{2}\) hyperplanes in \(\P^{n}\) in general position with respect to hyperplanes and to quadrics is ample modulo \(\mathscr{A}\).
\end{alphtheo}
\begin{proof}
  Consider an arrangement \(\mathscr{A}\) of \(d=n+1+k\) hyperplanes \(H^{0},\dotsc,H^{n+k}\) in general linear position.
  Choose homogeneous coordinates \(Z_0,\dotsc,Z_n\) of \(\P^{n}\) in such way that \(H^{0},\dotsc,H^{n}\) are given by the equations \(Z_{i}=0\), and that \(H^{n+j}\) is given by the equation
  \[
    a_0^j Z_0+a_{1}^{j} Z_{1}+\dotsb+a_{n}^{j} Z_{n}=0,
  \]
  for some complex coefficients \(a_i^j\), for \(j=1,\dotsc,k\). 

  In the dual projective space parametrizing hyperplanes, consider the coordinate points parametrizing \(H^{0},\dotsc,H^{n}\) and the points \((a_{0}^{j},\dotsc,a_{n}^{j})\) paramatrizing \(H^{n+1},\dotsc,H^{n+k}\).
  The arrangement \(\mathscr{A}\) is in general position with respect to hyperplanes if \((n+1)\) of these points never lie in a single hyperplane, and \(\mathscr{A}\) is in general position with respect to quadrics if \(\binom{n+2}{2}\) of these points never lie in a quadric.
  Recall that \(\binom{n+2}{2}-1\) hyperplanes in general linear position in \(\P^{n}\) determine a unique quadric in the dual projective space parametrizing hyperplanes. 

  Very concretely, in our setting, the arrangement \(\mathscr{A}\) is in general position with respect to hyperplanes when the minors (of any size) of the \((n+1)\times k\) coefficient matrix 
  \[
    A
    \bydef
    \left[\left[a_i^j\right]\right]_{\substack{0\leq i\leq n\\1\leq j\leq k}}
  \]
  are non-zero. 
  Moreover, for \(k\geq\binom{n+1}{2}\), the arrangement \(\mathscr{A}\) is in general position with respect to quadrics if all the maximal minors of the \(\binom{n+2}{2}\times(n+1+k)\) matrix of all degree \(2\) monomials in the equation coefficients are non-zero.
  Putting the squares in first position, and taking the coordinate points as the first \(n+1\) points, we get in particular that all maximal minors of the \(\binom{n+1}{2}\times k\) matrix of products \(a_{i_1}^ja_{i_2}^j\) (in lexicographic order)
  \[
    A_{[2]}
    \bydef
    \left[\left[
        a_{i_1}^ja_{i_2}^j
    \right]\right]_{\substack{0\leq i_1<i_2\leq n\\1\leq j\leq k}}
  \]
  are non-zero. We will use this fact at the end of the proof.

  Outside of \(\mathscr{A}\), one can work on the affine chart \(Z_0\neq0\). 
  The equations of the \(k\) last hyperplanes become
  \[
    a_0^j+a_{1}^{j} z_{1}+\dotsb+a_{n}^{j} z_{n}
    =
    0,
  \]
  in the inhomogeneous coordinates \(z_{j}\bydef Z_{j} / Z_{0}\).
  Then a local frame of the logarithmic tangent sheaf \(\Omega^{\vee}(\P^{n},\mathscr{A})\) around the origin in \(U_{0}\) is given by
  \(
  z_{1}\frac{\partial}{\partial z_{1}},
  \dotsc,
  z_{n}\frac{\partial}{\partial z_{n}}
  \),
  and if we denote (for obvious reason)
  \[
    z_{n+j}
    \bydef
    a_0^j+a_{1}^{j} z_{1}+\dotsb+a_{n}^{j} z_{n},
  \]
  a basis of the space of global sections \(H^{0}\big(\P^{n},\Omega(\P^{n},\mathscr{A})\bigr)\) is given by
  \[
    \frac{\diff z_{1}}{z_{1}},
    \dotsc,
    \frac{\diff z_{n}}{z_{n}},
    \frac{\diff z_{n+1}}{z_{n+1}},
    \dotsc,
    \frac{\diff z_{n+k}}{z_{n+k}}.
  \]

  The Kodaira map associated to \(\abs{\O_{\P(\Omega(\P^{n},\mathscr{A}))}(1)}\), maps a point
  \[
    (z,[\xi])
    =
    (z_{1},\dotsc,z_{n};[V_{1}z_{1}{\partial}/{\partial z_{1}}+\dotsb+V_{n}z_{n}{\partial}/{\partial z_{n}}])
    \in\P(\Omega(\P^{n},\mathscr{A})),
  \]
  to the point \(\varphi(\z,[\xi])\bydef[V_{1}:\dotso:V_{n}:\varphi^{1}(\boldsymbol{z},\boldsymbol{V}):\dotso:\varphi^{k}(\boldsymbol{z},\boldsymbol{V})]\in\P^{n+k-1}\), where:
  \[
    \varphi^{j}(\z,\boldsymbol{V})
    \bydef
    \frac{a_{1}^{j} V_{1}z_{1}+\dotsb+a_{n}^{j} V_{n}z^{n}}{a_0^j+a_{1}^{j} z_{1}+\dotsb+a_{n}^{j} z^{n}}.
  \]

  We will prove that under the assumptions of the theorem, \(\varphi\) gives an immersion. Then, we obtain the result by Lemma~\ref{lemm:bplus}.

  The coordinates \(V_i\) cannot be simultaneously zero. Regarding the symmetries of \(\varphi\), it is sufficient to prove that \(\varphi\) is immersive on one affine chart \(V_i\neq 0\). 
  Let us thus work on the chart \(V_1\neq0\), in affine coordinates \(v_i=V_i/V_1\), and in the affine chart ``\(Z_0\neq0\)'' in \(\P^{n+k-1}\).
  One has then:
  \[
    \varphi(\z,[\xi])
    =
    (v_{2},\dotsc,v_{n},\varphi^{1}(\boldsymbol{z},\boldsymbol{v}),\dotsc,\varphi^{k}(\boldsymbol{z},\boldsymbol{v}))
  \]
  and
  \[
    \varphi^{j}(\z,\boldsymbol{v})
    \bydef
    \frac{a_{1}^{j} z_{1}+a_{2}^{j} v_{2}z_{2}+\dotsb+a_{n}^{j} v_{n}z^{n}}{a_0^j+a_{1}^{j} z_{1}+\dotsb+a_{n}^{j} z^{n}}.
  \]

  The Jacobian matrix of \(\varphi\) with respect to the coordinates \((\z,\boldsymbol{v})\) is the matrix:
  \[
    \begin{pmatrix}
      0&\dotso&0&1&&0\\
      \vdots&^\cdot\cdot_\cdot&\vdots&&^\cdot\cdot_\cdot&\\
      0&\dotso&0&0&&1\\
       &&&*&\dotso&*\\
       &\partial\varphi^{i}/\partial z_{j}&&\vdots&^\cdot\cdot_\cdot&\vdots\\
       &&&*&\dotso&*\\
    \end{pmatrix}.
  \]
  Its rank is thus \(n-1+\rk(J)\), where \(J\bydef\big(\partial\varphi^{i}/\partial z_{j}\big)\). 

  Let us write by convention \(v_1=1\) from now on. 

  The simple computation 
  \[
    \partial\varphi^{j}/\partial z_{i}
    =
    \frac{
      a_0^ja_i^jv_i
      +
      a_{i}^{j}a_1^j (v_i-v_1)z_{1}+\dotsb+a_{i}^{j}a_n^j (v_i-v_n)z_{n}
    }
    {(a_0^j+a_{1}^{j} z_{1}+\dotsb+a_{n}^{j} z_{n})^2}.
  \]
  shows that this matrix can be written as a matrix product \(J=M\cdot A_{[2]}/(z_{n+1}\dotsm z_{n+k})^2\).
  Here the columns of \(M\) are
  \(M^{i}=v_{i}E_{i}\) for \(i=1,\dotsc,n\) and then
  \(M^{(i_1,i_2)}=(v_{i_1}-v_{i_2})(z_{i_2}E_{i_1}-z_{i_1}E_{i_2})\),
  where \(E_{1},\dotsc,E_{n}\) is the canonical basis of \(\C^{n}\) for \(1\leq i_1<i_2\leq n\) (in lexicographic order). E.g. for \(n=3\):
  \[
    M
    \bydef
    \begin{pmatrix}
      v_{1}&0&0&(v_{1}-v_{2})z_{2}&(v_{1}-v_{3})z_{3}&0\\
      0&v_{2}&0&(v_{2}-v_{1})z_{1}&0&(v_{2}-v_{3})z_{3}\\
      0&0&v_{3}&0&(v_{3}-v_{1})z_{1}&(v_{3}-v_{2})z_{2}\\

    \end{pmatrix}
  \]

  Points where \(\varphi\) is not an embedding are those where \(\rk(M\cdot A_{[2]})<n\). 
  We claim that \(\rk(M)=n\). If not, considering the first minor \(\abs{M^{1}\dotso M^{n}}\), one infers that at least one of the \(v_{i}\) has to be \(0\). Assume thus that \(v_{p+1},\dotsc,v_{n}\) are zero but \(v_{1},\dotsc,v_{p}\) are not. Note that \(p\geq 1\), since \(v_1=1\).
  The minor \(\abs{M^{1}\dotso M^{p} M^{(p,p+1)}\dotso M^{(p,n)}}\) is then \(v_{1}\dotsm v_{p}(-v_{p}z_{p})^{n-p+1}\) which is not zero since \(z_{p}\neq0\). This is a contradiction.
  Since \(k\geq\binom{n+1}{2}\), the matrix \(A_{[2]}\) has more columns than rows. By the general position assumption, it is of maximal row rank. Therefore \(\rk(J)=\rk(M\cdot A_{[2]})=\rk(M)=n\). This ends the proof of Theorem~\ref{theo:lognog}.
\end{proof}

\begin{rema}
  Observe that for \(6\) lines in \(\P^2\), we retrieve the generic condition brought out by Noguchi, by elementary linear algebra manipulations (in \cite{Nog86}'s convention, \(a_0^0=a_1^0=a_2^0=1\) and also \(a_{0}^{1}=a_{0}^{2}=1\)).
\end{rema}
\begin{rema}
  We do not really need the general position assumption for \(d>\binom{n+2}{2}\), but we only need that at least \(\binom{n+2}{2}\) of the \(d\) hyperplanes satisfy it.
\end{rema}

\begin{ques}
  For the critical degree \(d=\binom{n+2}{2}\),
  is there an obstruction to positivity of logarithmic cotangent bundles if all hyperplanes lie in a single quadric ?
\end{ques}

\section{Positivity of orbifold cotangent bundles}
\label{se:campana}
\subsection{Campana's orbifold category}
Before proceeding to the proof, let us first make some recall.

A \textsl{smooth orbifold pair} is a pair \((X,\Delta)\), where \(X\) is a smooth projective variety and where \(\Delta\) is a \(\Q\)-divisor on \(X\) with only normal crossings and with coefficients between \(0\) and \(1\).
In analogy with ramification divisors, it is very natural to write
\[
  \Delta
  =
  \sum_{i\in I}
  (1-\myfrac{1}{m_{i}}) \Delta_{i},
\]
with \textsl{multiplicities} \(m_{i}=a_i/b_i\) in \(\Q_{\geq1}\cup\Set{+\infty}\).
If \(b_{i}=0\), by convention \(a_{i}=1\).
The multiplicity \(1\) corresponds to empty boundary divisors. 
The multiplicity \(+\infty\) corresponds to reduced boundary divisors.
We denote \(\abs{\Delta}\bydef\sum_{i\in I}\Delta_i\) (it could be slightly larger than the support of \(\Delta\) because of possible multiplicities \(1\)).

Such pairs \((X,\Delta)\) are studied using their \textsl{orbifold cotangent bundles} (\cite{CP15}).
Following the presentation used notably in \cite{C15},  it is natural to define these bundles on certain Galois coverings, the ramification of which is partially supported on \(\Delta\).
A Galois covering \(\pi\colon Y\to X\) from a smooth projective (connected) variety $Y$ will be termed \textsl{adapted} for the pair \((X,\Delta)\) if
\begin{itemize}
  \item
    for any component \(\Delta_{i}\) of \(\abs{\Delta}\), \(\pi^{\ast}\Delta_{i}=p_{i}D_{i}\), where \(p_{i}\) is an integer multiple of \(a_{i}\) and \(D_{i}\) is a simple normal crossing divisor;
  \item
    the support of \(\pi^{\ast}\Delta+\mathrm{Ram}(\pi)\) has only normal crossings, and the support of the branch locus of \(\pi\) has only normal crossings.
\end{itemize}
There always exists such an adapted covering (\cite[Prop. 4.1.12]{Laz}).

Let \(\pi\colon Y\to X\) be a \(\Delta\)-adapted covering. For any point \(y\in Y\), there exists an open neighbourhood \(U\ni y\) invariant under the isotropy group of \(y\) in \(\Aut(\pi)\), equipped with centered coordinates \(w_{i}\) such that \(\pi(U)\) has coordinates \(z_{i}\) centered in \(\pi(y)\) and
\[
  \pi(w_{1},\dotsc,w_{n})
  =
  (z_{1}^{p_{1}},
  \dotsc,
  z_{n}^{p_{n}}),
\]
where \(p_{i}\) is an integer multiple of the coefficient \(a_{i}\) of \((z_{i}=0)\).
Here by convention, if \(z_{i}\) is not involved in the local definition of \(\Delta\) then \(a_{i}=b_{i}=1\).

If all multiplicities are infinite (\(\Delta=\abs{\Delta}\)), for any \(\Delta\)-adapted covering \(\pi\colon Y\to X\), we denote
\[
  \Omega(\pi,\Delta)
  \bydef
  \pi^{\ast}\Omega_{X}(\log \Delta).
\]
For arbitrary multiplicities, the \textsl{orbifold cotangent bundle} is defined to be the vector bundle \(\Omega(\pi,\Delta)\) fitting in the following short exact sequence:
\begin{equation}
  \label{eq:orbi_cotangent}
  0
  \to
  \Omega(\pi,\Delta)
  \hookrightarrow
  \Omega(\pi,\abs{\Delta})
  \stackrel{\mathrm{res}}{\longrightarrow}
  \bigoplus_{i\in I\colon m_{i}<\infty}
  \O_{\myfrac{\pi^{\ast}\Delta_{i}}{m_{i}}}
  \to
  0.
\end{equation}
Here the quotient is the composition of the pullback of the residue map
\[
  \pi^{\ast}\mathrm{res}
  \colon
  \pi^{\ast}\Omega_{X}(\log \abs{\Delta})
  \to
  \bigoplus_{i\in I\colon m_{i}<\infty}
  \O_{\pi^{\ast}\Delta_{i}}
\]
with the quotients
\(
\O_{\pi^{\ast}\Delta_{i}}
\twoheadrightarrow
\O_{\myfrac{\pi^{\ast}\Delta_{i}}{m_{i}}}
\)
(\cite[\textit{loc. cit.}]{C15}).

Alternatively, the sheaf of orbifold differential forms adapted to \(\pi\colon Y\to(X,\Delta)\) is the subsheaf
\(
\Omega(\pi,\Delta)
\subseteq
\Omega(\pi,\abs{\Delta})
\)
locally generated (in coordinates as above) by the elements
\[
  w_{i}^{\myfrac{p_{i}}{m_{i}}}
  \pi^{\ast}(\diff z_{i}/z_{i})
  =
  w_{i}^{-p_{i}(1-\myfrac{1}{m_{i}})}
  \pi^{\ast}(\diff z_{i}).
\]
Note that if the multiplicities \(m_i\)'s are integers and if the cover \(\pi\) is \textsl{strictly adapted} (i.e. \(p_i=m_i\)), then \(\Omega(\pi,\Delta)\) identifies with \(\Omega_Y\) via the differential map of \(\pi\).

\subsection{Orbifold positivity}
The direct image of the sheaf of \(\Aut(\pi)\)-invariant sections of \(S^N\Omega(\pi,\Delta)\)
\[
  S^{[N]}\Omega(X,\Delta)
  \bydef
  \pi_{\ast}((S^N\Omega(\pi,\Delta)))^{\Aut(\pi)}
  \subseteq
  S^N\Omega_{X}(\log\abs{\Delta}),
\]
is a subsheaf of logarithmic symmetric differentials which does not depend on the choice of \(\pi\).
Note that in almost all situations \(S^{[N]}\Omega(X,\Delta)\neq S^N\Omega(X,\Delta)\).
The sheaves \(S^{[N]}\Omega(X,\Delta)\) are independently defined and cannot be seen as symmetric powers.
One has merely a morphism \(S^pS^{[N]}\Omega(X,\Delta)\to S^{[pN]}\Omega(X,\Delta)\) given by multiplication.
However, the philosophy in the framework of Campana's orbifolds is to study orbifold pairs through adapted covers, and we will.

We would like to relate positivity properties of the orbifold cotangent bundle with some positivity properties of \(\Omega(\pi,\Delta)\), for some adapted cover \(\pi\). 
The definition for bigness is quite clear. 
\begin{defi}
  We say that \((X,\Delta)\) has a \textsl{big cotangent bundle} if \(\Omega(\pi,\Delta)\) is big for some (hence for all) adapted cover \(\pi\).
  Equivalently, the orbifold cotangent bundle of the pair \((X,\Delta)\) is \textsl{big} if for some/any ample integral divisor \(A\subseteq X\), there exists an integer \(N\) such that
  \(H^0(X,S^{[N]}\Omega(X,\Delta)\otimes A^\vee) \neq \{0\}\).
\end{defi}

To define ampleness, we will use augmented base loci, or rather their natural projections. 
In the spirit of~\cite{MU}, in which augmented base loci of vector bundles are studied, we define the orbifold augmented base locus of the cotangent bundle to the pair \((X,\Delta)\), as follows.
Recall that the base locus of a vector bundle \(E\) is defined in~\cite{MU} as
\[
  \Bs(E)
  \bydef
  \left\{x \in X\middle/H^0(X,E) \to E_x \text{ is not surjective}\right\}.
\]
\begin{defi}
  The \textsl{orbifold augmented base locus} of \(\Omega(X,\Delta)\) is
  \[
    \BB_+(\Omega(X,\Delta))
    \bydef
    \bigcap_{p/q\in \Q}\bigcap_{N>0}
    \Bs(S^{[Nq]}\Omega(X,\Delta)\otimes (A^\vee)^{\otimes Np}),
  \]
  for an integral ample divisor \(A\to X\).
\end{defi}

Before proceeding to the definition, observe first the following.
For an adapted cover \(\pi\colon Y\to (X,\Delta)\), we use the notation \(Y'\bydef\P(\Omega(\pi,\Delta))\to Y\).
\begin{prop}
  \label{prop:pi_inv}
  Over $X\setminus \abs{\Delta}$,
  the image of the augmented base locus $\BB_+(\O_{Y'}(1))$ by the natural projection \(Y'\twoheadrightarrow Y\twoheadrightarrow X\) does not depend on \(\pi\). Indeed, it actually coincides with the orbifold augmented base locus \(\BB_+(\Omega(X,\Delta))\vert_{X\setminus\abs\Delta}\).
\end{prop}
\begin{proof}
  \begin{enumerate}
    \item
      We claim that for any adapted cover \(\pi\), in order to compute the augmented base locus of \(\Omega(\pi,\Delta)\), it is sufficient to consider \(\Aut(\pi)\)-invariant sections. 

      Notice first that because of the relative ampleness of $\O_{Y'}(1)$, one can assume that the ample line bundle in the definition of $\BB_+(\O_{Y'}(1))$ is the pull-back of an ample line bundle on $X$, and in particular is  invariant under  \(\Aut(\pi)\). 
      Then observe that $\BB_+(\O_{Y'}(1))$ is \(\Aut(\pi)\)-invariant.
      Indeed, for any global section \(\sigma\) and for any element of the Galois group \(\alpha\), the Galois transform \(\sigma^\alpha\) is also a global section.
      We deduce that for each orbit of \(\Aut(\pi)\), either all points are in the augmented base locus, or none. 

      Let $\BB^G_+(\O_{Y'}(1))$ denote the base locus obtained by considering only \(\Aut(\pi)\)-invariant sections. 
      If $v \in \BB_+(\O_{Y'}(1))$, then obviously $v\in\BB^G_+(\O_{Y'}(1))$.
      Conversely, consider $v \not \in \BB_+(\O_{Y'}(1))$. By the preceding considerations, the (finite) orbit of $v$ stays outside $\BB_+(\O_{Y'}(1))$. 
      By Noetherianity, $\BB_+(\O_{Y'}(1))$ can be realized as a single base locus.
      One can then find a divisor in the associated linear system that avoids all the points in the orbit of \(v\).
      In other words, one can find a global section \(\sigma\) which does not vanish at any point of the orbit of $v$. Moreover, this section can be made invariant after multiplication by its Galois conjugates. To conclude, $\BB_+(\O_{Y'}(1)) = \BB^G_+(\O_{Y'}(1))$.
    \item
      Now remark that there is a natural morphism $\pi^*S^{[N]}\Omega(X,\Delta) \to S^N (\Omega(\pi,\Delta))$ which is an injection of sheaves and an isomorphism outside 
      $\abs{\Delta}$. Combining with the preceding equality of base loci, one obtains that $\BB_+(\O_{Y'}(1))$ has a projection on $X\setminus \abs{\Delta}$ which depends only on the sheaves $S^{[N]}\Omega(X,\Delta)$. Namely (reasoning as in~\cite{MU}), it is the restriction of \(\BB_+(\Omega(X,\Delta))\).
      \qedhere
  \end{enumerate}
\end{proof}

There are many interesting situations where one cannot expect global ampleness of \(\Omega(\pi,\Delta)\) but where bigness is not sufficient for applications (see below). Therefore, we shall introduce an intermediate positivity property.

\begin{defi}
  We say that \((X,\Delta)\) has an \textsl{ample cotangent bundle modulo boundary} if its orbifold augmented base locus is contained in the boundary. 
\end{defi}

  Equivalently, \((X,\Delta)\) has an ample cotangent bundle modulo boundary, if for some (hence for all) adapted cover \(\pi\), the orbifold cotangent bundle \(\Omega(\pi,\Delta)\) is ample modulo the \(\Aut(\pi\))-invariant closed subset living over the boundary. This definition will be used in practice.

\begin{rema}
  As a consequence of Proposition~\ref{prop:pi_inv}, the ``ampleness modulo boundary'' of \(\Omega(\pi,\Delta)\) does not depend on \(\pi\).
  Ampleness of orbifold cotangent bundles has been recently studied in the PhD thesis of Tanuj Gomez where it is shown by a different method that for strictly adapted covers ramifying exactly on \(\abs\Delta\), the (global) ampleness of \(\Omega(\pi,\Delta)\) does not depend on the cover.
  It would be interesting to check to which extent ampleness of \(\Omega(\pi,\Delta)\) is equivalent to the triviality of \(H^q(X,S^{[N]}\Omega(X,\Delta)\otimes A^{\otimes p})\), for some \(A\) ample, any \(p,q>0\), and \(N\) large enough.
\end{rema}

\begin{rema}
  In general, one cannot expect that there exists a strictly adapted covering ramifying exactly over the boundary divisor. But if \(\pi\colon Y\to(X,\Delta)\) is a strictly adapted cover ramifying exactly over \(\Delta\), a convenient way to prove that the orbifold cotangent bundle \(\Omega(X,\Delta)\) is ample modulo boundary is to prove that the orbifold cotangent bundle \(\Omega(\pi,\Delta)\simeq\Omega_Y\) is ample modulo its ramification locus and it is actually equivalent.
\end{rema}

\subsection{Obstructions to orbifold positivity}
Positivity of cotangent bundles of projective manifolds or log-cotangent bundles of pairs has been investigated by many authors (see \textit{e.g.} \cite{Deb05,Xie18,BDa18,Den20,CR20,Etesse,Nog86,BDe18}). In the orbifold setting, much less seems to be known. Nevertheless, results of \cite{Sommese} can be interpreted as the study of ampleness of orbifold cotangent bundles associated to orbifolds $(\P^2, \Delta)$ corresponding to arrangements of lines in $\P^2$. In particular, \cite[Theo.~4.1]{Sommese} characterizes exactly which arrangements have ample orbifold cotangent bundles. An interesting consequence of this result is that when the orbifold $(\P^2, \Delta)$ is smooth (i.e. when the lines are in general position), the orbifold cotangent bundle is \emph{never} ample. 
This is due to the following fact. Let $C$ be any irreducible component of $\pi^{-1}(\abs{\Delta})$ then ${\Omega_Y}_{|C} \cong \Omega_C \oplus N_C^*$ (\cite[p.~217]{Sommese}), and $\deg N_C^*=-C^2 \leq 0$.
In other words, each component of the boundary carries a negative quotient.

This can be generalized as follows.
\begin{lemm}
  Let \((\P^n,\Delta)\) be a smooth orbifold pair with integer (or infinite) coefficients. Then, for any strictly adapted covering \(\pi\) the cotangent bundle \(\Omega(\pi,\Delta)\) has negative quotients supported on each boundary component with finite multiplicity, and trivial quotients supported on each boundary component with infinite multiplicity.
\end{lemm}
\begin{proof}
  Let \(\Delta=(1-1/m_1)\Delta_1+\Delta'\), where the multiplicity of \(\Delta_1\) in \(\Delta'\) is zero.

  If \(m_1=\infty\), the residue exact sequence
  \(
  \Omega(\pi,\abs{\Delta'})
  \hookrightarrow
  \Omega(\pi,\abs{\Delta})
  \twoheadrightarrow
  \O_{\pi^\ast\Delta_1}
  \)
  restricts to
  \[
    \Omega(\pi,\Delta')
    \hookrightarrow
    \Omega(\pi,\Delta)
    \twoheadrightarrow
    \O_{\pi^\ast\Delta_1}.
  \]
  We get the sought trivial quotient on \(\abs{\pi^\ast\Delta_1}\)

  If \(m_1<\infty\), let \(D_1\bydef\myfrac{\pi^{\ast}\Delta_{1}}{m_{1}}\). Note that this is a reduced divisor. By~\eqref{eq:orbi_cotangent}, one has:
  \[
    \Omega(\pi,\Delta)
    \hookrightarrow
    \Omega(\pi,\abs{\Delta})
    \twoheadrightarrow
    \O_{D_1}
    \oplus
    \bigoplus_{i\in I\colon m_{i}<\infty}
    \O_{\myfrac{\pi^{\ast}\Delta_{i}'}{m_{i}}},
  \]
  and
  \[
    \Omega(\pi,\Delta')
    \hookrightarrow
    \Omega(\pi,\abs{\Delta})
    \twoheadrightarrow
    \O_{m_1D_1}
    \oplus
    \bigoplus_{i\in I\colon m_{i}<\infty}
    \O_{\myfrac{\pi^{\ast}\Delta_{i}'}{m_{i}}}.
  \]
  One infers
  \[
    \Omega(\pi,\Delta')
    \hookrightarrow
    \Omega(\pi,\Delta)
    \twoheadrightarrow
    \O_{m_1D_1}
    \diagup
    \O_{D_1}.
  \]
  Let \(\mathcal{I}\) denote the ideal sheaf of \(D_1\) in \(Y\).
  The quotient above is isomorphic to \(\mathcal{I}\diagup\mathcal{I}^{m_1}\). Composing with the quotient \(\mathcal{I}/\mathcal{I}^{m_1}\twoheadrightarrow \mathcal{I}/\mathcal{I}^2\simeq \mathcal N_{D_1}^*\), we deduce that \(\Omega(\pi,\Delta)\) has a negative quotient supported on \(\abs{D_1}\) (and namely the conormal bundle of \(D_1\)).
\end{proof}

Therefore, starting with smooth orbifold pairs associated to hyperplane arrangements in projective spaces, the best one can hope for is ampleness modulo the boundary.

\section{Ampleness modulo boundary of the orbifold cotangent bundle}
\label{se:orbinog}
This section is devoted to prove the following extension of Theorem~\ref{theo:lognog}.
\begin{alphtheo}
  \label{theo:orbinog}
  The orbifold cotangent bundle along an arrangement \(\mathscr{A}\) of \(d\geq \binom{n+2}{2}\) hyperplanes in \(\P^{n}\) in general position with respect to hyperplanes and to quadrics, with multiplicities \(m\geq 2n+2\), is ample modulo \(\mathscr{A}\).
\end{alphtheo}
We keep the setting and notation of Sect.~\ref{se:lognog}.
\begin{rema}
  One can accept different multiplicities for the hyperplanes.
  Indeed, lowering all multiplicities to the lowest one (still assumed at least \(2n+2\)), one fits in the setting of Theorem~\ref{theo:orbinog}. But the augmented base locus of the orbifold cotangent sheaf can only increase by doing this operation. (Note that the projectivized orbifold cotangent bundles are isomorphic outside of the boundary). Applying the same reasoning, one can also treat infinite multiplicities.
\end{rema}

\subsection{Fermat covers}
\label{sse:FC}
Considering the \(k\) linear relations between the hyperplanes:
\[
  H_{n+j}=a_0^jH_0+\dotsb+a_n^jH_n,
\]
we identify the projective space \(\P^n\) with the linear subspace of \(\P^N\bydef\P^{n+k}\) cut out by the \(k\) linear equations
\[
  Z_{n+j}=a_0^jZ_0+\dotsb+a_n^jZ_n.
\]
in homogeneous coordinates \(Z_0,\dotsc,Z_N\), for \(N\bydef n+k\).
We also define the complete intersection \(Y\) in \(\P^N\) of the \(k\) Fermat hypersurfaces
\[
  Z_{n+j}^m=a_0^jZ_0^m+\dotsb+a_n^jZ_n^m.
\]
The map \(\pi\colon[Z_i]\mapsto[Z_i^m]\) realizes \(Y\) as a cover of \(\P^n\) ramifying exactly over the hyperplanes \(H_i\), with multiplicity \(m\). In other words, \(Y\) is a (strictly) adapted cover of the orbifold pair \((\P^n,\Delta)\), where \(\Delta=(1-1/m)(H_0+\dotsb+H_{N})\). We call \(\pi\colon Y\to(\P^n,\Delta)\) the \textsl{Fermat cover} of \((\P^n,\Delta)\).

The cotangent bundle of the orbifold pair \((\P^n,\Delta)\) is ample modulo boundary
when the cotangent bundle of its Fermat cover is ample modulo its ramification locus.

An obvious obstruction to ampleness of the cotangent bundle is the presence of rational lines.
The following remark gives another nice justification that we need to take at least \(2n+1\) hyperplanes in order to hope for the orbifold cotangent bundle to be ample.
\begin{rema}
  \label{rema:standardlines}
  Recall that each Fermat hypersurface of degree \(m\) (without zero coefficient) in \(\P^{n+1}\) contains a \(n-2\) dimensional family of ``standard'' lines. The standard lines can be described as follows. To each partition of the set \(\{0,\dotsc, n + 1\}\) in \(r\) subsets with cardinalities \(i_1 ,\dotsc, i_r \geq 2\), there is a rational map \(\P^{n+1}\dashrightarrow\P^{i_1-1}\times\dotsb\times \P^{i_r-1}\), the fibers of which are linear subspaces \(\P^{r-1}\) . Its restriction to the Fermat hypersurface yields a rational map onto a product of lower dimensional Fermat hypersurfaces (of total dimension \(n+2-2r\)). Each fiber of this map contains a \((2r - 4)\)-dimensional family of lines, which are called standard.
\end{rema}
\begin{lemm}
  There is no standard line in a generic complete intersection of \(k\) Fermat hypersurfaces in \(\P^{n+k}\) iff \(k\geq n\).
\end{lemm}
\begin{proof}
  Now, we consider a complete intersection of \(k \geq 2\) Fermat hypersurfaces in \(\P^{n+k}\) with generic coefficients, and we consider only partitions of \(\{0,\dotsc,n+k\}\) in subsets with cardinalities at least \(1 + k\) (otherwise the intersection of the complete intersection with the linear subspace would be generically empty). There is no nontrivial such partition as soon as \(n + 1 + k < 2(1 + k)\), i.e. \(k \geq n\).
\end{proof}

\subsection{Explicit symmetric differentials on Fermat covers}
In \cite{Bro16}, Brotbek has described a way to produce global twisted symmetric differentials on complete intersections \(Y\) in \(\P ^{N}\).
The following is a slight adaptation to the particular setting of Fermat covers of~\cite{Bro16} (see also~\cite{Xie18,Demailly}) ; this could appear not so obvious due to some redaction shortcuts.
We could have made the proof (slightly) more heuristic with an approach involving a \(N\times(N+1)\) matrix in the spirit of \cite{Bro16}, but here we have prefered compactness.

\begin{lemm}
  For any subset of pairwise distinct integers \(\{j_1,\dotsc,j_n\}\) in \(\{n+1,\dotsc,n+k\}\), there is a global section of \(S^{n}\Omega_{Y}(2n+1-m)\) given on \(Z_0\neq 0\) by:
  \[
    \sigma
    \bydef
    \begin{vmatrix}
      a_{1}^{j_1-n}(z_1z_{j_1}'-z_1'z_{j_1})&\dotso&a_{n}^{j_1-n}(z_nz_{j_1}'-z_n'z_{j_1})\\
      \vdots&&\vdots\\
      a_{1}^{j_n-n}(z_1z_{j_n}'-z_1'z_{j_n})&\dotso&a_{n}^{j_n-n}(z_nz_{j_n}'-z_n'z_{j_n})
    \end{vmatrix}
    \otimes
    Z_0^{2n+1-m},
  \]
  where \(z_i\bydef Z_i/Z_0\) denote the standard affine coordinates on the chart \(Z_0\neq0\).
\end{lemm}
\begin{proof}
  We would like to underline that the most interesting part of the lemma is the ``extra vanishing'' of order \(m-1\) that we shall now explain (see \cite[Sect. 12D]{Demailly} for an analog construction for higher order jet differentials).

  The proof relies on the following very basic fact of linear algebra.  Consider a \(n\times(n+1)\) matrix such that the sum of the columns is zero, then (up to sign) all its \(n\times n\) minors are equal. This is more or less Cramer's rule.
  Let us denote by \(\det_{\overline c}\) the minor obtained by removing column \(c\) from such a matrix.

  For any \(j=n+1,\dotsc,n+k\), one has:
  \begin{align*}
    a_0^{j-n}+a_1^{j-n}z_1^m+\dotsb+a_n^{j-n} z_n^m= z_{j}^m &&\text{and}&& a_1^{j-n}z_1^{m-1}z_1'+\dotsb+a_n^{j-n} z_n^{m-1}z_n'= z_{j}^{m-1}z_{j}'.
  \end{align*}
  Therefore:
      \[
        \begin{pmatrix}
          a_{0}^{j_1-n}(z_0z_{j_1}'-z_0'z_{j_1})&\dotso&a_{n}^{j_1-n}(z_nz_{j_1}'-z_n'z_{j_1})\\
          \vdots&&\vdots\\
          a_{0}^{j_n-n}(z_0z_{j_n}'-z_0'z_{j_n})&\dotso&a_{n}^{j_n-n}(z_nz_{j_n}'-z_n'z_{j_n})
        \end{pmatrix}
        \begin{pmatrix}
          z_0^{m-1}\\\vdots\\z_n^{m-1}
        \end{pmatrix}
        =
        0.
      \]
  where we denote \(z_0\bydef1\) and \(z_0'\bydef0\) for convenience.
  Observe that one has
  \[
    \sigma
    =
    \det_{\overline 0}
        \begin{pmatrix}
          a_{0}^{j_1-n}(z_0z_{j_1}'-z_0'z_{j_1})&\dotso&a_{n}^{j_1-n}(z_nz_{j_1}'-z_n'z_{j_1})\\
          \vdots&&\vdots\\
          a_{0}^{j_n-n}(z_0z_{j_n}'-z_0'z_{j_n})&\dotso&a_{n}^{j_n-n}(z_nz_{j_n}'-z_n'z_{j_n})
        \end{pmatrix}
    \otimes
    Z_0^{2n+1-m}.
  \]
  One immediately infers some alternative expressions of \(\sigma\) on the intersections \((Z_0Z_i\neq0)\), in which the sought ``extra'' vanishing appear. 
  \begin{enumerate}
    \item
      On \((Z_0Z_1\neq0)\), we can use \(\det_{\overline 1}\) instead of \(\det_{\overline0}\) (the same reasoning holds for \(Z_2,\dotsm,Z_n\)) :
      \[
        \sigma
        =
        -\frac{z_0^{m-1}}{z_1^{m-1}}
        \det_{\overline 1}
        \begin{pmatrix}
          a_{0}^{j_1-n}(z_0z_{j_1}'-z_0'z_{j_1})&\dotso&a_{n}^{j_1-n}(z_nz_{j_1}'-z_n'z_{j_1})\\
          \vdots&&\vdots\\
          a_{0}^{j_n-n}(z_0z_{j_n}'-z_0'z_{j_n})&\dotso&a_{n}^{j_n-n}(z_nz_{j_n}'-z_n'z_{j_n})
        \end{pmatrix}
        \otimes
        Z_0^{2n+1-m}.
      \]
      Let \(y_i=z_i/z_1\) denote the standard affine coordinates on \((Z_1\neq0)\). Recall that \(z_0=1\). We get the following expression for \(\sigma\):
      \[
        \sigma
        =
        -
        \det_{\overline 1}
        \begin{pmatrix}
          a_{0}^{j_1-n}(y_0y_{j_1}'-y_0'y_{j_1})&\dotso&a_{n}^{j_1-n}(y_ny_{j_1}'-y_n'y_{j_1})\\
          \vdots&&\vdots\\
          a_{0}^{j_n-n}(y_0y_{j_n}'-y_0'y_{j_n})&\dotso&a_{n}^{j_n-n}(y_ny_{j_n}'-y_n'y_{j_n})
        \end{pmatrix}
        \otimes
        Z_1^{2n+1-m}.
      \]
      Here we use
      \[
        (z_iz_j'-z_i'z_j)=z_1^2(y_iy_j'-y_i'y_j).
      \]
    \item
      Let us now consider the intersection \((Z_0Z_{n+1}\neq0)\). The same reasoning holds for \(Z_{n+2},\dotsc,Z_{n+k}\).
      Here we have to enlarge the matrix by considering
      \[
        \begin{pmatrix}
          a_0^{1}z_0&\dotso&a_n^{1}z_n&1\\
          a_{0}^{j_1-n}(z_0z_{j_1}'-z_0'z_{j_1})&\dotso&a_{n}^{j_1-n}(z_nz_{j_1}'-z_n'z_{j_1})&0\\
          \vdots&&\vdots&\vdots\\
          a_{0}^{j_n-n}(z_0z_{j_n}'-z_0'z_{j_n})&\dotso&a_{n}^{j_n-n}(z_nz_{j_n}'-z_n'z_{j_n})&0
        \end{pmatrix}
        \begin{pmatrix}
          z_0^{m-1}\\\vdots\\z_n^{m-1}\\-z_{n+1}^{m}
        \end{pmatrix}
        =
        0.
      \]
      One has
  \[
    \sigma
    =
    (-1)^n
    \det_{\overline 0}
        \begin{pmatrix}
          a_0^{1}z_0&\dotso&a_n^{1}z_n&1\\
          a_{0}^{j_1-n}(z_0z_{j_1}'-z_0'z_{j_1})&\dotso&a_{n}^{j_1-n}(z_nz_{j_1}'-z_n'z_{j_1})&0\\
          \vdots&&\vdots&\vdots\\
          a_{0}^{j_n-n}(z_0z_{j_n}'-z_0'z_{j_n})&\dotso&a_{n}^{j_n-n}(z_nz_{j_n}'-z_n'z_{j_n})&0
        \end{pmatrix}
    \otimes
    Z_0^{2n+1-m}.
  \]
  Using \(\det_{\overline{n+1}}\) instead of \(\det_{\overline 0}\), one gets the alternative expression:
      \[
        \sigma
        =
        -
        \frac{z_0^{m-1}}{z_{n+1}^m}
        \det
        \begin{pmatrix}
          a_0^{j_1-n}z_0&\dotso&a_n^{j_1-n}z_n\\
          a_{0}^{j_1-n}(z_0z_{j_1}'-z_0'z_{j_1})&\dotso&a_{n}^{j_1-n}(z_nz_{j_1}'-z_n'z_{j_1})\\
          \vdots&&\vdots\\
          a_{0}^{j_n-n}(z_0z_{j_n}'-z_0'z_{j_n})&\dotso&a_{n}^{j_n-n}(z_nz_{j_n}'-z_n'z_{j_n})
        \end{pmatrix}
        \otimes
        Z_0^{2n+1-m}.
      \]
      Let \(y_i=z_i/z_1\) denote the standard affine coordinates on \((Z_{n+1}\neq0)\). We get:
      \[
        \sigma
        =
        -\det
        \begin{pmatrix}
          a_0^{j_1-n}y_0&\dotso&a_n^{j_1-n}y_n\\
          a_{0}^{j_1-n}(y_0y_{j_1}'-y_0'y_{j_1})&\dotso&a_{n}^{j_1-n}(y_ny_{j_1}'-y_n'y_{j_1})\\
          \vdots&&\vdots\\
          a_{0}^{j_n-n}(y_0y_{j_n}'-y_0'y_{j_n})&\dotso&a_{n}^{j_n-n}(y_ny_{j_n}'-y_n'y_{j_n})
        \end{pmatrix}
        \otimes
        Z_{n+1}^{2n+1-m}.
      \]
  \end{enumerate}

  This ends the proof.
\end{proof}
\begin{rema}
  Note that the zero locus of \(\sigma\) does not depend on \(m\). However, one will need \(m> 2n+1\) to get a global symmetric differential vanishing on an ample divisor.
\end{rema}
\subsection{Augmented base locus}
Let \(V\subset Y\) be the \(\Aut(\pi)\)-invariant open subset living above \(X\setminus\abs\Delta\). In other words \(V=(Z_0\dotsm Z_N\neq 0)\).
\begin{theo}
  When \(m>2n+1\), the projection of the augmented base locus of \(\O_{\P(\Omega_{Y})}(1)\) does not intersect the open \(V\).
\end{theo}
\begin{proof}
  In this proof, we use repeatedly that we work on \(Z_0\dotsm Z_N\neq0\), and we will not necessarily mention it anymore.

  Let us denote:
  \[
    B
    \bydef
    \begin{pmatrix}
      a_{0}^{1}(z_0z_{n+1}'-z_0'z_{n+1})&\dotso&a_{n}^{1}(z_nz_{n+1}'-z_n'z_{n+1})\\
      \vdots&&\vdots\\
      a_{0}^{k}(z_0z_{n+k}'-z_0'z_{n+k})&\dotso&a_{n}^{k}(z_nz_{n+k}'-z_n'z_{n+k})
    \end{pmatrix},
  \]
  where \(z_0,\dots,z_N\) are the standard extrinsic affine coordinates on \((Z_i\neq0)\) (for some \(i\in\{0,\dotsc,n\}\)), and where \(z_0',\dotsc,z_N'\) are the standard extrinsic homogeneous coordinates on \(\P(\Omega_{Y})\subset\P(\Omega_{\P^N})\). By convention \(z_i=1,z_i'=0\).

  \begin{enumerate}
    \item
  The augmented base locus is contained in the locus where \( \rk B < n \).
  Indeed, since the first column is a non-zero linear combination of the last \(n\) columns, it is equivalent to say that the rank of the \(n\) last column is less than \(n\). But by the previous lemma, \(n\times n\)-minors in the last \(n\) columns are global sections of \(S^n\Omega_Y(2n+1-m)\).
  Here, it is also useful to notice that \(\O(1)\) is relatively ample on \(\P(\Omega_Y)\).
  Therefore, one can define the augmented base locus of \(\O(1)\) using the pullback of an ample line bundle on \(Y\). 

\item
  In the spirit of the proof in the logarithmic case, we will write \(B\) as a product involving the matrix \(A_{[2]}\).
  Denote \(b_i^j\) the coefficients of the matrix \(B\). For \(j=1,\dotsc,k\), using the equations of \(\P(\Omega_Y)\), one has:
  \[
    z_{n+j}^{m-1}b_{i_1}^j
    =
    z_{n+j}^{m-1}a_{i_1}^j(z_{i_1} z_{n+j}'-z_{i_1}' z_{n+j})
    =
    \sum_{i_2=0}^n a_{i_1}^ja_{i_2}^j (z_{i_1}z_{i_2}'-z_{i_1}'z_{i_2}).
  \]
  Therefore, 
  \[
    B = \mathrm{diag}(1/z_{n+1}^{m-1},\dotsc,1/z_{n+k}^{m-1})\cdot A_{[2]}^T\cdot W,
  \]
  where \(W\) is a \(\binom{n+1}{2}\times(n+1)\)-matrix,
  row \((i_1,i_2)\) of which is \((z_{i_1}z_{i_2}'-z_{i_1}'z_{i_2})(E_{i_1}-E_{i_2})\) (denoting \(E_0,\dotsc,E_n\) the canonical basis of \(\C^{n+1}\)).
  For example, for \(n=2\):
  \[
    W
    =
    \begin{pmatrix}
      (z_0z_1'-z_0'z_1)&
      (z_1z_0'-z_1'z_0)&
      0
      \\
      (z_0z_2'-z_0'z_2)&
      0&
      (z_2z_0'-z_2'z_0)
      \\
      0&
      (z_1z_2'-z_1'z_2)&
      (z_2z_1'-z_2'z_1)
    \end{pmatrix}.
  \]

  One infers that \(\rk B = \rk A_{[2]}^TW\). Moreover, under the assumption that \(A_{[2]}\) is full row rank (which also means that \(A_{[2]}^T\) is full column rank), one has \(\rk(A_{[2]}^TW)=\rk W\). Hence:
  \[
    \rk B
    =
    \rk W.
  \]

\item
  Now, we claim that \(W\) is of rank at least \(n\), from which one deduces the result of the theorem, by the first two points of the proof.

  Indeed, we will exhibit a non-zero \(n\times n\) minor in \(W\).
  We work with the standard affine coordinates on \((Z_0\neq0)\).
  For shortness we will write \(w_{i_1,i_2}\) for \(z_{i_1}z_{i_2}'-z_{i_1}'z_{i_2}\).
  If \(z_1'=\dotsb=z_n'=0\), using the equations of \(\Omega_{Y}\), one would immediately get that all first derivatives are simultaneously zero, which is not possible. Assume therefore that at least one of these first derivatives is non zero, say \(z_1'=w_{0,1}\).
  For \(i=2,\dotsc,n\), one has:
  \(z_i w_{0,1} = (z_1 w_{0,i}-z_0 w_{1,i})\).
  As a consequence, at least one of \(w_{0,i}\) or \(w_{1,i}\) is non-zero.
  Let us call it \(w_{\star,i}\) for convenience.
  Recall that the rows of \(W\) are indexed by couples \((i_1<i_2)\in\{0,\dotsc,n\}^2\).
  Consider the \(n\times n\) minor made of columns \(1,\dotsc,n\) and of rows \((0,1)\), \((\star,2),\dotsc,(\star,n)\). It is
\[
  \begin{vmatrix}
    w_{1,0}&0&\dotso&\dotso&0\\
    *&w_{2,\star}&\ddots&&\vdots\\
    \vdots&0&\ddots&\ddots&\vdots\\
    \vdots&\vdots&\ddots&\ddots&0\\
    *&0&\dotso&0&w_{n,\star}
  \end{vmatrix}
  =
  (-1)^nw_{0,1}w_{\star,2}\dotsm w_{\star,n}
  \neq
  0.
\]
This proves our claim and therefore ends the proof.\qedhere
\end{enumerate}
\end{proof}

Theorem~\ref{theo:orbinog} is then a plain corollary, because the Fermat cover \(\pi\colon Y\to(\P^n,\Delta)\) is an adapted cover such that \(\Omega(\pi,\Delta)\simeq\Omega_{Y}\).

\begin{rema}
  One could deduce Theorem~\ref{theo:lognog} from the same proof, but we have the feeling that the natural role of the coefficient matrix \(A_{[2]}\) would be less highlighted in this way.
\end{rema}

\begin{rema}
  The proof above may look disappointingly simple, but it is actually the synthesis of very hard computational explorations. The matrix \(A_{[2]}\), brought out by the logarithmic case, is the key of the proof and it was a turning point when we were able to involve it in the proof for \(n=2\). All barriers quickly came down after that. We invite the reader to forget its existence for the fun and to try to find some genericity condition for \(\mathscr{A}\) already in the cases \(n=2\) (or \(n=3\) for the most daring) !
  \end{rema}

\section{Bigness of the orbifold cotangent bundle with low multiplicities}
\label{se:orbibig}
We are not able yet to generalize the strategy of Noguchi to the full orbifold category. Indeed, it seems very difficult to produce explicit global sections for \emph{low multiplicities}, even with a lot of components in the boundary. This is quite surprising in view of Theorem~\ref{theo:orbibig}, that we recall below.
\begin{alphtheo}
  \label{theo:orbibig}
  For \(n\geq 2\), the orbifold cotangent bundle along an arrangement \(\mathscr{A}\) of \(d\geq 2n(\frac{2n}{m-2}+1)\) hyperplanes in \(\P^{n}\) with multiplicity \(m\geq 3\) is big.
\end{alphtheo}

For \(n=2\), it was proved in~\cite{CDR20} that for \(m \geq 2\) the orbifold cotangent bundle \(\Omega(\P^2,\Delta)\) is big if \(d \geq 11\).
Here we generalize this statement to higher dimension for any multiplicity \(m \geq 3\). 

\begin{rema}
  \(m=3\) in any dimension is really difficult. This is illustrated by the following vanishing theorem proven in~\cite{CDR20}.
  If \(D\) is a (reduced) smooth divisor (with an \emph{arbitrary large degree}) in \(\P^{n}\) and if \(m\leq n\), then there is no non-zero global orbifold symmetric differential for the pair \((\P^{n},(1-1/m)D)\).
  Actually, there is even no non-zero global jet differential of any order (higher jet order analogs of symmetric differentials).
\end{rema}

\begin{proof}[Proof ot Theorem~\ref{theo:orbibig}]
  The proof relies on a theorem by Brotbek~\cite{Bro14} improved by Coskun and Riedl~\cite{CR20} on cotangent bundles of complete intersections, and on our use of Fermat covers.

  We keep the setting and notation of previous sections.
  Consider the Fermat cover \(\pi\colon Y\to (\P^n,\Delta)\) where \(\Delta\) is the orbifold divisor \(\Delta\bydef\sum_{i=0}^{n+k}(1-1/m)H_{i}\) on \(\P^{n}\).
  Showing that \(\Omega(\P^{n},\Delta)\) is big is equivalent to showing that \(\Omega(\pi,\Delta)\simeq\Omega_{Y}\) is big.
  In order to prove the bigness of \(\Omega_{Y}\), we apply Theorem~2.7 in~\cite{CR20} which gives that a smooth complete intersection of dimension \(n\) in \(\P^N\) and type \((d_1,\dots, d_c)\), with \(c\geq n\), has big cotangent bundle if
  \[
    d_i \geq \frac{4n^2}{N-2n+1}+2.
  \]
  In our situation, \(N=n+k\), and the complete intersection has type \((m,\dotsc,m)\).
\end{proof}

\begin{rema}
  The methods used in~\cite{CDR20} (Riemmann--Roch) and in~\cite{CR20} (Morse inequalities) do not provide any explicit global symmetrc differential. Hence the existence of a lot a global sections does not provide any precise geometric information on the augmented base locus. On the counterpart, the orbifold multiplicity in Theorem~\ref{theo:orbibig} is extremely low, and there is no genericity assumption on \(\mathscr{A}\).
\end{rema}

\section{Applications to complex hyperbolicity}
\label{se:hyperb}
\def\barI{\smash{\overline I}\phantom{I}}
\subsection{Entire curves in Fermat covers}
Hyperbolicity properties of Fermat hypersurfaces have been studied by several people. 
One can find in~\cite[Example~3.10.21]{Kob98} the following result.
\begin{theo}[Kobayashi]
  \label{theo:FCK}
  Consider the Fermat hypersurface of degree \(m\) 
  \[
    F(n,m)
    \bydef
    \{z_0^m+\dots+z_{n+1}^m=0\} \subseteq \P^{n+1}.
  \]
  If \(m \geq (n+1)^2 \) then every entire curve \(f\colon \C \to F(n,m)\) lies in a linear subspace of dimension at most \(\lfloor n/2 \rfloor\).
\end{theo}
The proof of~\cite{Kob98} consists in using the fact that \(F(n,m)\) is a Fermat cover of \(\P^n\) ramified over \((n+2)\) hyperplanes \(H_i\) with multiplicity \(m\). Then the result is a consequence of the truncated defect of Cartan (see~\cite[3.B.42]{Kob98}) which gives the linear degeneracy of orbifold entire curves \(f\colon\C \to (\P^n, \sum_{i=0}^{n+1}(1-1/m)H_i)\) provided that \(\sum_{i=0}^{n+1}(1-n/m)^+>n+1\). 

In~\cite[Ex.~11.20]{Dem97}, algebraic degeneracy of entire curves in \(F(n,m)\) is also obtained using jet differentials. It gives the same degree estimate but not the second assertion on the linear subspace of dimension \(\leq \lfloor n/2 \rfloor\) containing the image of the entire curve.

\begin{rema}
  The dimension of the linear subspace in Theorem~\ref{theo:FCK} is (at least) almost optimal.
  In the setting of Remark~\ref{rema:standardlines}, if instead of considering rational lines, one now considers entire curves as in Theorem~\ref{theo:FCK}, and one takes \(r=\lfloor n/2\rfloor\), one infers that the dimension of the linear subspaces needed for some curves in Theorem~\ref{theo:FCK} cannot be less than \(\lfloor n/2\rfloor-1\).
\end{rema}

As a consequence of Theorem~\ref{theo:orbinog}, we obtain the following result on hyperbolicity of Fermat covers as introduced in Section~\ref{sse:FC}.
\begin{alphtheo}
  \label{theo:fermathyp}
  The Fermat cover associated to an arrangement \(\mathscr A\) of \(d\geq\binom{n+2}{2}\) hyperplanes in \(\P^{n}\) in general position with respect to hyperplanes and to quadrics, with ramification \(m\geq 2n+2\) is Kobayashi-hyperbolic.
\end{alphtheo}
\begin{proof}
  Let \(\pi\colon Y\to(\P^n,\Delta)\) be the associated Fermat cover.
  Since \(Y\) is compact, it is sufficient to prove that \(Y\) is Brody hyperbolic. Let \(f\colon \C \to Y\) be an entire curve. Theorem~\ref{theo:orbinog} implies that \(f(\C)\) is contained in the ramification locus of \(\pi\colon Y\to(\P^n,\Delta)\). Now we remark that the ramification locus has a natural structure of Fermat cover associated to an induced arrangement \(\mathscr{A}^1\) of \(d\geq \binom{n+2}{2}\)-1 hyperplanes in \(\P^{n-1}\) with multiplicity \(m\). 
  Up to removing some members, one can still assume that this arrangement is in general position with respect to hyperplanes and to quadrics, by Lemma~\ref{lemm:gal.pos} below.
  Using inductively Theorem~\ref{theo:orbinog}, we obtain the hyperbolicity of the Fermat cover associated to \(\mathscr{A}\).
\end{proof}
\begin{lemm}
  \label{lemm:gal.pos}
  Let \(\mathscr{A}\) be an arrangement of \(d\geq\binom{n+2}{2}\) hyperplanes in \(\P^n\) in general position with respect to hyperplanes and to quadrics. 
  Let \(\mathscr{A}^{\barI}\) be the arrangement obtained by removing \(\abs{I}\) hyperplanes, indexed by \(I\).
  In \(\bigcap_{i\in I}H^i\simeq\P^{n-\abs{I}}\),
  there is a subarrangement \(\mathscr{A}'\) of \(\mathscr{A}^{\barI}\) with at least \(\binom{n-\abs{I}+2}{2}\) members, such that \(\mathscr{A}'\cap \P^{n-\abs{I}}\) is in general position with respect to hyperplanes and to quadrics.
\end{lemm}
\begin{proof}
  The situation being symmetric, we can safely assume that we are in the setting and notation of Sect.~\ref{se:lognog}, and that we have removed the \(\abs{I}\) last coordinate hyperplanes.
  Clearly, any subarrangement is still in linear position, because if \(n+1-\abs{I}\) hyperplanes of \(\mathscr{A}^{\barI}\) would satisfy a single linear equation in \(\P^{n-\abs{I}}\) then these hyperplanes together with the \(\abs{I}\) last coordinate hyperplanes would satisfy the same linear equation in \(\P^{n}\).
  Now, we want to prove that at least one subarrangement is also in general position with respect to quadrics. For any \(\binom{n+2}{2}-\abs{I}\) hyperplanes in \(\mathscr{A}^{\barI}\), containing the \(n+1-\abs{I}\) first coordinate hyperplanes, we have seen that the general position with respect to quadrics is equivalent to the non-vanishing of the determinant of the matrix \(A_ {[2]}\). We split the rows of \(A_{[2]}\) in two blocks: those involving the \(\abs{I}\) last coordinates, and those that does not.
  By Laplace expansion with respect to these blocks, there is at least one minor in the second block that is not zero.
  We take the \(\binom{n-\abs{I}+1}{2}\) hyperplanes corresponding to the columns involved in one such minor, and the \((n-\abs{I}+1)\) first hyperplane coordinates, and we get \(\binom{n-\abs{I}+2}{2}\) hyperplanes in general position with respect to hyperplanes and to quadrics in \(\P^{n-\abs{I}}\).
\end{proof}

A complete intersection of general Fermat hypersurfaces cannot be reduced to a Fermat cover. Moreover, we cannot use openness of ampleness in families without additional efforts (see~\cite{BDe18}). However, it is most likely that the results obtained in the present work for Fermat covers would generalize to complete intersections of general Fermat hypersurfaces. We even think that this problem should be accessible using the technics involved in this work. As an example, we were able to prove that general complete intersection surfaces of Fermat type in \(\P^{2+k}\) have ample cotangent bundles modulo ramification for \(k\geq3\), under some explicit algebraic condition on their coefficients. For higher dimensions, computations become tedious, and we would probably need to use (explicit) resultant theory in order to conclude.
This is far beyond the scope of this work. Let us hence formulate the expected results as questions.
\begin{ques}
  Do general \(n\)-dimensional complete intersections of Fermat type with sufficiently large codimension (e.g. \(k\geq\binom{n+1}{2}\)) and ramification order (e.g. \(m\geq 2n+2\)) have ample cotangent bundles modulo their ramification loci? 
\end{ques}
\begin{ques}
  When are these general complete intersections of Fermat type Kobayashi-hyperbolic?
\end{ques}

\subsection{An orbifold Brody's criterion}
Let \((X,\Delta)\) be an orbifold with  
$\Delta
=
\sum_{i\in I}
(1-\myfrac{1}{m_{i}}) \Delta_{i}$.
Following \cite{CW09} it is natural to define (Kobayashi) hyperbolicity of \((X,\Delta)\) considering holomorphic maps $h\colon\mathbb{D}\to (X,\Delta)$ from the unit disk \(\mathbb D\) to \(X\) satisfying the two conditions:
\begin{itemize}
  \item $h(\mathbb{D}) \not \subset |\Delta|$.
  \item ${\mult}_x(h^*\Delta_{i}) \geq m_i$ for all $i$ and $x\in \mathbb{D}$ with $h(x) \in |\Delta_{i}|$.
\end{itemize}
Such maps are called \textsl{orbifold} maps \(\mathbb{D}\to(X,\Delta)\). 
Orbifold entire curves \(\C\to(X,\Delta)\) are defined \textit{mutadis mutandis}.

Then, one defines the \textsl{orbifold Kobayashi pseudo-distance} $d_{(X,\Delta)}$ as the largest pseudo-distance on $X \setminus \lfloor \Delta \rfloor$ such that every orbifold map from the unit disk is distance-decreasing with respect to the Poincar\'e distance on the unit disk. A pair \((X,\Delta)\) is said \textsl{Kobayashi-hyperbolic} if $d_{(X,\Delta)}$ is a distance on $X \setminus \lfloor \Delta \rfloor$.

Besides, a pair \((X,\Delta)\) is said \textsl{Brody-hyperbolic} if it does not admit any non-constant orbifold entire curve \(\C\to(X,\Delta)\). Kobayashi-hyperbolicity implies Brody hyperbolicity. 
Brody's theorem characterizes Kobayashi-hyperbolicity in terms of Brody-hyperbolicity. We will now give an orbifold version of this result.

We have the following proposition which slightly refines \cite{CW09}.
\begin{prop}
  Let $\big(X, \Delta\bydef\sum_{i=0}^{d}(1-1/m_i)\Delta_i\big)$ be an orbifold. Assume that a sequence of orbifold maps  $h_p\colon \mathbb{D} \to (X, \Delta)$ from the unit disk converges locally uniformly to a holomorphic map $h\colon \mathbb{D} \to X$. 
  Let \(X_h\bydef\bigcap_{h(\mathbb{D})\subseteq\Delta_i}\Delta_i\subseteq X\), 
  and let \(\Delta_h\bydef \sum_{h(\mathbb{D})\not\subseteq\Delta_j}(1-1/m_j)\Delta_j\cap X_h\).
  Then, $h$ is an orbifold map $\mathbb{D}\to(X_h, \Delta_h)$.
\end{prop}
\begin{proof}
  Suppose that $h(0) \in |\Delta|$. Consider a neighbourhood \(V\) of $h(0)$ in $X$ such that $\abs\Delta\cap V$ is locally defined by a holomorphic function $\prod f_i$, where $f_i=0$ defines $\Delta_i\cap V$. 
  If \(h(D)\not\subseteq \Delta_j\), one can assume that $f_j\circ h$ has no zero in $V$ except at $0$.

  Apply the classical theorem of Rouché to a sequence of holomorphic function $\{f_j \circ h_p\}$.
  For all sufficiently large $p$ the multiplicity at $0$ of $f_j\circ h$ equals the sum of all multiplicities of all zeroes in $V$ of $f_j\circ h_p$.
  Therefore this multiplicity is at least $m_j$ because $h_p$ are orbifold maps. 
\end{proof}

As an immediate consequence, reasoning exactly as in~\cite[Sect.~13]{CW09}, we obtain the following result.
\begin{theo}[orbifold Brody's criterion]
  \label{theo:brodyorb}
  Consider a smooth orbifold pair 
  \[
    \left(X,\Delta\bydef\sum_{i=0}^{d}(1-1/m_i)\Delta_i\right).
  \]
  For a subset $I$ of $\Set{0,\dotsc,d}$,
  let \(X_I\bydef\cap_{i\in I}\Delta_i\),
  and let \(\Delta_{\barI}\bydef\sum_{j\not\in I} (1-1/m_j) \Delta_j\cap X_I\).
  If all pairs \((X_I,\Delta_{\barI})\) are Brody-hyperbolic, then the pair \((X,\Delta)\) is Kobayashi-hyperbolic.
\end{theo}

\subsection{Orbifold hyperbolicity}
Now we are in position to derive 
from Theorem~\ref{theo:orbinog}
an hyperbolicity result for the orbifold pair 
\[
  \left(\P^n, \Delta\bydef\sum_{i=0}^{d}\left(1-\frac{1}{m}\right)H_i\right).
\]
We will use Fermat cover, in the opposite direction as Kobayashi did in Theorem~\ref{theo:FCK}. Remark however that this orbifold hyperbolicity is not directly implied by Theorem~\ref{theo:fermathyp}, because orbifold curves do not lift in general to the Fermat cover. Nevertheless, techniques introduced in~\cite{CDR20} will permit to use it.
Indeed, Corollary~3.7 of~\cite{CDR20} yields the following Proposition.
\begin{prop}[Fundamental vanishing theorem]
  \label{prop:ample}
  Let \((X,\Delta)\) be a smooth orbifold pair. Then any orbifold entire curve is contained in  $\BB_+(\Omega(X,\Delta)).$
\end{prop}

We obtain the following.
\begin{alphtheo}
  \label{theo:orbihyp}
  Consider an arrangement \(\mathscr A\) of \(d\) hyperplanes \(H_1,\dotsc,H_d\) in \(\P^n\) in general position with respect to hyperplanes and to quadrics, with respective orbifold multiplicities \(m_i\), and the associated orbifold divisor \(\Delta\bydef\sum_{i=0}^d(1-1/m_i)\cdot H_i\).
  If \(d\geq\binom{n+2}{2}\) and \(m_i\geq 2n+2\), then the orbifold pair \((\P^n,\Delta)\) is hyperbolic.
\end{alphtheo}
\begin{proof}
  By Theorem~\ref{theo:brodyorb}, if \((\P^n,\Delta)\) is not hyperbolic, then either there exists a non-constant orbifold entire curve \(f\colon \C \to (\P^N, \Delta)\) or an orbifold curve in the boundary divisor.
  Theorem~\ref{theo:orbinog} and Proposition~\ref{prop:ample} imply that all orbifold entire curves \(f\colon\C\to(\P^n,\Delta)\) are constant. So, we are left with the second possibility. In this case, according to Theorem~\ref{theo:brodyorb}, \(f\) is a non-constant orbifold map with respect to an orbifold structure $(\P^{n-\abs I},\Delta_{\barI})$ induced by the arrangement \(\mathscr{A}^{\barI}\) of \(\binom{n+2}{2}-|I|\geq\binom{n+2-\abs{I}}{2}\) hyperplanes.
  We conclude by induction, using Lemma~\ref{lemm:gal.pos}.
\end{proof}
\begin{rema}
  It follows that \((\P^n,\Delta)\) is Brody-hyperbolic.
  Actually, one can exclude the existence of non-constant \textsl{orbifold correspondences} on varieties with orbifold cotangent bundles that are ample modulo boundary. Cf.~\cite{CDR20} for a definition. These are the morphisms that one would naturally consider to generalize entire curves in the orbifold category (and these are much more numerous).
\end{rema}

\bigskip

\paragraph{Acknowledgments}
L.D. and E.R. would like to thank \emph{Joël Merker} for interesting discussions on explicit orbifold sections and around resultant which helped a lot to find the right attack angle for our problem.

L.D. would like to thank \emph{Henri Guenancia} for his help on augmented base loci and particularly around Lemma~\ref{lemm:bplus}, which plays an important role in the reformulation of Noguchi's result. L.D. would also like to thank \emph{Mikhail Zaidenberg} for many interesting discussions over the years and for making him aware of Theorem~\ref{theo:optimal}.
These interactions took place during the conference Alkage hosted by \emph{Jean-Pierre Demailly}, which gave L.D. a great opportunity to present a preliminary version of this work to a distinguished audience.

E.R. would like to thank \emph{Stefan Kebekus} and \emph{Tanuj Gomez} for discussions on the positivity of orbifold cotangent bundles, and \emph{Eric Riedl} for discussions on bigness of cotangent bundles.

L.D. would like to thank \emph{Frédéric Han} for identifying the condition of general position with respect to quadrics.

Lastly, L.D. would like to warmly thank \emph{Damian Brotbek} for introducing him in much detail to his work on explicit symmetric differential forms during the supervision of his postdoc in Strasbourg, and for the fruitful collaboration that followed.

\bibliographystyle{smfalpha}
\bibliography{noguchi}
\vfill
\end{document}